\documentclass[11pt]{amsart}
\usepackage{amscd,amssymb}
\usepackage[all]{xy}
\usepackage[colorlinks,plainpages,backref,urlcolor=blue]{hyperref}

\newtheorem{theorem}{Theorem}[section]
\newtheorem{corollary}[theorem]{Corollary}
\newtheorem{lemma}[theorem]{Lemma}
\newtheorem{prop}[theorem]{Proposition}

\theoremstyle{definition}
\newtheorem{definition}[theorem]{Definition}
\newtheorem{example}[theorem]{Example}
\newtheorem{remark}[theorem]{Remark}
\newtheorem*{ack}{Acknowledgments}

\newcommand{\Z}{\mathbb{Z}}

\newcommand{\C}{\mathbb{C}}

\newcommand{\PP}{\mathbb{P}}

\renewcommand{\k}{\Bbbk}
\newcommand{\RR}{{\mathcal R}}
\newcommand{\VV}{\mathcal{V}}
\newcommand{\WW}{\mathcal{W}}

\newcommand{\m}{{\mathfrak{m}}}

\newcommand{\T}{{\mathcal{T}}}

\newcommand{\wX}{\widetilde{X}}

\DeclareMathOperator{\rank}{rank}
\DeclareMathOperator{\gr}{gr}
\DeclareMathOperator{\im}{im}

\DeclareMathOperator{\id}{id}
\DeclareMathOperator{\ab}{{ab}}

\DeclareMathOperator{\ch}{char}

\DeclareMathOperator{\Hom}{{Hom}}
\DeclareMathOperator{\Tor}{{Tor}}

\DeclareMathOperator{\ann}{{ann}}

\DeclareMathOperator{\supp}{supp}
\DeclareMathOperator{\specm}{Spec}

\DeclareMathOperator{\nil}{Nil}
\DeclareMathOperator{\Tors}{{Tors}}

\newcommand{\same}{\Longleftrightarrow}
\newcommand{\surj}{\twoheadrightarrow}
\newcommand{\inj}{\hookrightarrow}

\newcommand{\isom}{\xrightarrow{\,\simeq\,}}

\def\set#1{{\left\{#1\right\}}}

\def\dot{\mathchar"013A}  
\newcommand{\hdot}{{\raise1pt\hbox to0.35em{\Large $\dot$\!}}} 

\begin{document}
\date{February 25, 2014}

\title[Jump loci in the equivariant spectral sequence]{%
Jump loci in the equivariant spectral sequence}

\author[Stefan Papadima]{Stefan Papadima$^1$} 
\address{Simion Stoilow Institute of Mathematics, 
P.O. Box 1-764,
RO-014700 Bucharest, Romania}
\email{Stefan.Papadima@imar.ro}
\thanks{$^1$Partially supported by the Romanian Ministry of
National Education, CNCS-UEFISCDI, grant PNII-ID-PCE-2012-4-0156}

\author[Alexander~I.~Suciu]{Alexander~I.~Suciu$^2$}
\address{Department of Mathematics,
Northeastern University,
Boston, MA 02115, USA}
\email{a.suciu@neu.edu}
\thanks{$^2$Partially supported by NSF grant DMS--1010298 
and NSA grant H98230-13-1-0225}



\begin{abstract}
We study the homology jump loci of a chain complex over 
an affine $\k$-algebra.  When the chain complex is the first 
page of the equivariant spectral sequence associated to
a regular abelian cover of a finite-type CW-complex, we 
relate those jump loci to the resonance varieties 
associated to the cohomology ring of the space. 
As an application, we show that vanishing resonance 
implies a certain finiteness property for the completed 
Alexander invariants of the space. We also show that
vanishing resonance is a Zariski open condition, on a
natural parameter space for connected, finite-dimensional 
commutative graded algebras.
\end{abstract}
\maketitle
\tableofcontents

\section{Introduction}
\label{sect:intro}

\subsection{Overview}
\label{intro1}
The study of the homology groups of abelian covers 
goes back to the 1920s, when J.W.~Alexander 
introduced his eponymous knot polynomial. 
Given a knot in $S^3$, let $X^{\ab}\to X$ be the 
universal abelian cover of its complement. 
The Alexander polynomial of the knot, then, 
is the order of $H_1(X^{\ab},\Z)$, a finitely 
generated, torsion module over the group-ring 
$\Z[H_1(X,\Z)]\cong \Z[t^{\pm 1}]$. 

More generally, consider a connected, finite-type 
CW-complex $X$, with fundamental group  $\pi=\pi_1(X,x_0)$. 
The homology groups of $X^{\ab}$, with coefficients in $\C$, 
are finitely generated modules over the Noetherian ring $\C[\pi_{\ab}]$.   
As shown by Dwyer and Fried in \cite{DF}, the support loci 
of the Alexander invariants $H_i(X^{\ab},\C)$ completely 
determine the (rational) homological finiteness properties 
of its regular, free abelian covers.  This result was recast 
in \cite{PS-plms} in terms of the characteristic varieties 
of $X$, defined as the jump loci for homology with 
coefficients in rank $1$ local systems on our space.  
For a detailed treatment of these topics, and further 
developments, we refer to \cite{Su, SYZ}. 

We revisit here this theory from a more general point of view. 
The key new ingredient in our approach is the equivariant 
spectral sequence from \cite{PS-tams}.  Using techniques 
from commutative algebra, we establish a tight 
connection between the homology jump loci of the first page 
of this spectral sequence, on one hand, and the resonance 
varieties associated to the cohomology ring of the space, 
on the other hand.  In turn, this connection allows us to 
prove an infinitesimal analogue of the Dwyer--Fried finiteness 
test for the completed homology groups of abelian Galois covers. 

\subsection{The equivariant spectral sequence}
\label{intro2}
Let $\nu\colon \pi\surj G$ be an epimorphism from 
the fundamental group of $X$ to a (finitely generated) 
abelian group $G$, and let $X^{\nu}\to X$ be the 
corresponding regular cover. The homology groups 
of $X^{\nu}$ with coefficients in an 
algebraically closed field $\k$ are modules 
over the group-ring $\k{G}$. 
 
Let $J$ be the augmentation ideal of $\k{G}$. The associated 
graded ring, $S=\gr_J(\k{G})$, then, is a finitely generated 
$\k$-algebra. 
As shown in \cite{PS-tams}, there is a spectral sequence over $S$
that converges to the $J$-adic completion of $H_*(X^{\nu},\k)$.
The first page of this spectral sequence is a chain complex   
$E$, with terms the finitely generated $S$-modules 
$E_i=S\otimes_{\k} H_i(X,\k)$ and differentials 
that can be written in terms of the co-multiplication in $H_*(X,\k)$ 
and the induced homomorphism $\nu_*\colon H_1(\pi,\k)\to H_1(G,\k)$. 

The homology jump loci of $E$ are the 
subsets $\VV^i_d(E)$ of the maximal spectrum $\specm(S)$ 
consisting of those maximal ideals $\m$ for which the 
$\k$-vector space $H_i(E\otimes_S S/\m)$ has dimension 
at least $d$. 
To understand these sets in a more geometric fashion, 
consider the projection $\phi\colon G\surj \bar{G}$ 
onto the maximal free-abelian quotient of $G$.   
Then, as we show in Lemma \ref{lem:sbars}, 
the set $\specm(S)$ may be identified  
with $H^1(\bar{G},\k)$. Thus, we may view the 
sets $\VV^i_d(E)$ as subvarieties of the 
$\k$-vector space $H^1(\bar{G},\k)$.

The above definition of homology jump loci works for arbitrary 
chain complexes $E$ over a finitely generated $\k$-algebra $S$. 
In the case when $E$ is the cellular chain complex of the 
universal abelian cover $X^{\ab}$, with coefficients in $\k$, 
and $S=\k[\pi_{\ab}]$, the corresponding jump loci 
(also known as the characteristic varieties of $X$) 
are Zariski closed subsets of the character group 
$\Hom(\pi_{\ab},\k^{\times})$. These sets, which were 
introduced by Green and Lazarsfeld in \cite{GL}, are  
extremely useful in a variety of settings, see for 
instance \cite{DPS-duke, PS-plms, DP-ann, Su, SYZ}.

\subsection{Resonance varieties}
\label{intro3}
To state our main result, we need to recall one more concept. 
Using the cohomology algebra $A=H^*(X,\k)$ as input, we define 
the resonance varieties of $X$ as the sets $\RR^i_d(X,\k)$ 
consisting of those square-zero elements $a\in A^1$ for which 
the $\k$-vector space $H^i(A,a)$ has dimension at least $d$, 
where $(A,a)$ is the cochain complex with terms $A^i$ and 
differentials given by left-multiplication by $a$.  As we show 
in Corollary \ref{cor:vwr}, it is enough to assume that $X$ 
has finite $k$-skeleton, in order to conclude that the sets 
$\RR^i_d(X,\k)$ are Zariski closed, for all $i\le k$ and $d>0$. 

The vanishing property for resonance plays a crucial role in
Theorem \ref{thm:intro1}\eqref{i2} and Corollary \ref{cor:intro2} below. 
In sections \ref{subsec:parameter} and \ref{subsec:vanish}, 
which are largely self-contained and independent of the 
rest of the paper, we analyze this vanishing property 
for arbitrary commutative graded algebras.  More precisely, 
we show in Theorem \ref{thm:gen res} that, if $B$ is a connected, 
finite-dimensional, graded $\k$-vector space, and $P_B$ 
is the parameter space for all commutative graded algebras 
whose underlying graded vector space is $B$, then the 
subset of those cgas whose  resonance vanishes 
is a Zariski open subset of $P_B$.

\subsection{Resonance and the completed Alexander invariants}
\label{intro4}
We are now ready to state our main result, which relates 
the resonance varieties of a space to the finiteness properties 
of the completed homology groups of its abelian Galois covers. 
(Proofs for the two parts of this result will be given in 
Theorems \ref{thm:cv res} and \ref{thm:nustar res}, respectively.)

\begin{theorem}
\label{thm:intro1}
Let $X$ be a connected CW-complex, let $\nu\colon \pi\surj G$ 
be a homomorphism onto a finitely generated abelian group $G$, 
and let $\bar\nu^*\colon H^1(\bar{G},\k) \inj H^1(\pi,\k)$ be 
the homomorphism induced in cohomology by $\bar\nu=\phi\circ \nu$. 

\begin{enumerate}
\item \label{i1}
If $X$ is of finite type, and $E$ is the first page of the 
corresponding equivariant spectral sequence, then 
$\bar\nu^*(\VV^i_d(E))=\im(\bar\nu^*)\cap \RR^i_d(X,\k)$, 
for all $i\ge 0$ and $d>0$.
\\[-8pt]
\item  \label{i2}
Suppose $X$ has finite $k$-skeleton for some $k\ge 1$, 
and, for each $0\le i\le k$, the 
linear subspace $\im(\bar\nu^*)$ intersects the resonance variety 
$\RR^i_1(X,\k)$ at most at $0$.  Then, for each $0\le i\le k$, 
the completion of $H_i(X^{\nu},\k)$ with respect to the 
$J$-adic filtration is finite-dimensional.
\end{enumerate}
\end{theorem}

As an application, we obtain the following corollary.

\begin{corollary}
\label{cor:intro2}
If $X$ has finite $k$-skeleton, and 
all the resonance varieties $\RR^i_1(X,\k)$ with $i\le k$ 
vanish or are empty, then the completions of the Alexander 
invariants $H_i(X^{\ab},\k)$ are finite-dimensional $\k$-vector 
spaces, for all $i\le k$.  
\end{corollary}

This corollary generalizes Theorem C(2) from \cite{DP-ann}, 
a result proved in that paper in the case when $k$ equals $1$ 
and the coefficient field is $\C$, using a different approach.
It should be noted that, even in this situation, the 
Alexander invariants themselves may well be infinite-dimensional. 

An interesting example where Corollary \ref{cor:intro2} applies 
arises when $X=\T_g$, the Torelli space for curves of genus $g\ge 4$.
Working over $\C$, the vanishing resonance property for $k=1$ 
is proved in \cite[Theorem 4.4]{DP-ann}.  Furthermore, 
Theorem B from \cite{DHP}, in  genus $g\ge 6$, together 
with a recent result of Hain from \cite{H}, in genus $g=4$ or $5$, 
give a non-trivial computation of the completion of $H_1(\T_g^{\ab},\C)$.

\section{The homology jump loci of a chain complex}
\label{sect:jump}

In this section we introduce the support varieties and 
homology jump loci of a chain complex over a finitely 
generated $\k$-algebra, and study some of their properties.

\subsection{Maximal spectrum and supports}
\label{subsec:spec supp}

We start by reviewing some standard notions from commutative 
algebra, as they can be found, for instance, in \cite{Ei}. 

Fix a ground field $\k$. Let $S$ be a commutative, 
finitely generated $\k$-algebra (also known as an 
affine $\k$-algebra), and let $\specm(S)$ be the 
maximal spectrum of $S$.  This set comes 
endowed with the Zariski topology, whereby 
a subset $F$ is closed if and only if  there is an ideal 
$\mathfrak{a}\subseteq S$ such that $F$ equals 
$V(\mathfrak{a})=\set{\m \mid \m\supseteq  \mathfrak{a}}$, 
the ozero set of the ideal $\mathfrak{a}$.  

Now suppose $\k$ is algebraically closed. Then 
$S/\m=\k$, for every maximal ideal $\m$, 
and we have a natural identification 
\begin{equation}
\label{eq:mspec}
\specm(S)=\Hom_{\text{$\k$-alg}} (S,\k),
\end{equation}
under which a maximal ideal $\m\subset S$ 
corresponds to the $\k$-algebra morphism  
$\rho\colon S \to S/\m=\k$. 

Denote by $S_{\m}$ the localization of the ring $S$ 
at the maximal ideal $\m$. Clearly, the above morphism 
$\rho$  factors through a ring morphism 
$\rho_{\m}\colon S_{\m}\to S_{\m}/\m S_{\m}$.
Given an $S$-module $M$, denote by $M_{\m}$ its localization 
at the maximal ideal $\m$; then $M_{\m}$ acquires in a natural 
way the structure of an $S_{\m}$-module. Define the {\em support}\/ 
of $M$ as 
\begin{equation}
\label{eq:supp}
\supp(M)=\{ \m \in \specm(S) \mid M_{\m}\ne 0\}.
\end{equation}

If $M$ is a finitely generated $S$-module, the 
support of $M$ is a Zariski closed subset of $\specm(S)$, since 
\begin{equation}
\label{eq:vann}
\supp(M) = V( \ann ( M) ),
\end{equation}
where $\ann(M) \subseteq S$ is the annihilator $M$. 

\subsection{Support loci}
\label{subsec:supp jump}

Let $E=(E_{\bullet},d)$ be a non-negatively graded chain 
complex over $S$; in other words, a sequence of $S$-modules 
$\{E_i\}_{i\ge 0}$ and $S$-linear maps between them, 
\begin{equation}
\label{eq:cc}
\xymatrix{\cdots \ar[r] & E_{i} \ar^{d_i}[r] & E_{i-1} \ar[r] & \cdots \ar[r] & 
E_1 \ar^{d_1}[r] & E_0 \ar[r] & 0},
\end{equation}
satisfying $d_{i}\circ d_{i+1}=0$.  Evidently, the homology groups of 
the chain complex, $H_i(E)=\ker d_i/\im d_{i+1}$, are again $S$-modules. 

\begin{definition}
\label{def:supp var}
The {\em support varieties}\/ of the $S$-chain complex $E$ are 
the supports of the exterior powers of the homology modules 
$H_*(E)$:
\begin{equation}
\label{eq:supp cc}
\WW^i_d(E)  = \supp \Big(\bigwedge^d H_i(E)\Big). 
\end{equation}
\end{definition}

These subsets of $\specm(S)$ are defined for all integers $i$ and all 
non-negative integers $d$; they are empty if $i<0$ and $d>0$.  
Furthermore, for each $i\ge 0$, we have a nested sequence
\begin{equation}
\label{eq:wwid}
\specm(S)=\WW^i_0(E) \supseteq \WW^i_1(E) \supseteq  
\WW^i_2(E) \supseteq\cdots .
\end{equation} 
It is readily checked that these sets depend 
only on the chain-homotopy equivalence class 
of the $S$-chain complex $E$. 

Now suppose $E$ is a chain complex of finitely 
generated $S$-modules.  Then the sets $\WW^i_d(E)$ are 
Zariski closed subsets of $\specm(S)$, for all integers $i$ 
and $d\ge 0$. Indeed, if $E_i$ is a finitely generated $S$-module, 
then $H_i(E)$ is also finitely generated, and so are all its exterior 
powers; thus, the assertion follows from formula \eqref{eq:vann}.

\subsection{Homology jump loci}
\label{subsec:jumps}

Let $E$ be a chain complex of $S$-modules as in \eqref{eq:cc}. 
From now on, we will assume that the coefficient field $\k$ is 
algebraically closed, to insure that the residue fields $S/\m$ 
are isomorphic to $\k$, for all $\m\in \specm(S)$.

\begin{definition}
\label{def:jump var}
The {\em homology jump loci}\/ of the $S$-chain complex $E$ are 
defined as
\begin{equation}
\label{eq:jump cc}
\VV^i_d(E)  = \{ \m \in \specm(S) \mid \dim_{\k} H_i(E\otimes_S S/\m) \ge d\}.
\end{equation}
\end{definition}

As before, these sets are defined for all integers $i$ and all 
non-negative integers $d$; they are empty if $i<0$ and $d>0$;   
and, for each $i\ge 0$, they form a nested sequence
\begin{equation}
\label{eq:vvid}
\specm(S)=\VV^i_0(E) \supseteq \VV^i_1(E) \supseteq  
\VV^i_2(E) \supseteq\cdots .
\end{equation} 
Furthermore, the sets $\VV^i_d(E)$ depend 
only on the chain-homotopy equivalence class  
of  $E$. Under some mild restrictions on the chain 
complex $E$, its homology jump loci 
are Zariski closed subsets.  The next lemma   
 makes this statement 
more precise.

\begin{lemma}
\label{lem:zcv}
Suppose $E$ is a chain complex of free, finitely 
generated $S$-modules, and $\k$ is algebraically closed.  
Then the sets $\VV^i_d(E)$ are 
Zariski closed subsets of $\specm(S)$, for all integers $i$ 
and $d\ge 0$. 
\end{lemma}

\begin{proof}
By definition, a maximal ideal $\m\in \specm(S)$ 
belongs to the set $\VV^i_d(E)$ if and only if 
$\rank d_{i+1}(\m) + \rank d_{i}(\m) \le c_i - d$, 
where $c_i=\rank_{S} E_i$ and $d_{i}(\m)=d_{i}\otimes_S S/\m$. 
Hence, $\VV^i_d(E)$ is the zero-set of the ideal generated 
by all minors of size $c_{i}-d+1$ of the block-matrix 
$d_{i+1}\oplus d_{i}$. 
\end{proof}

As the next example indicates, the freeness assumption 
from Lemma \ref{lem:zcv} is crucial for the conclusion to hold, 
even in the presence of the finite-generation assumption.  
We refer to \cite[Example 9.5]{DP12} for other, more 
exotic examples, where both the freeness and the finite-generation 
assumptions are violated.

\begin{example}
\label{ex:aug}
Let $\k$ be an algebraically closed field, and 
let $S=\k[x]$.  Consider the chain complex 
of (finitely-generated) $S$-modules  
$E\colon S \xrightarrow{\,\epsilon\,} \k$,  
where $E_0=\k$, 
viewed as a trivial $S$-module, and 
$\epsilon$ is the $\k$-algebra map 
given by $\epsilon(x)=0$. It is readily verified that 
$\VV^1_1(E) =\k\setminus \{0\}$, which, of course, 
is not a Zariski closed subset of $\specm(S)=\k$.
\end{example}

\subsection{Comparing the two sets of loci}
\label{subsec:ww vs vv}

We are particularly interested in the sets 
$\WW^i_1(E)  = \supp H_i(E)$ and 
$\VV^i_1(E)  = \{ \m  \mid  H_i(E\otimes_S S/\m) \ne 0\}$. 
The next theorem (which, as we shall see in 
\S\ref{cor:vv ww}, generalizes a result from \cite{PS-plms}),  
establishes a comparison between these two types of sets.  

\begin{theorem}
\label{thm:vv ww}
Let $S$ be a finitely generated algebra over an algebraically 
closed field $\k$, and 
let $E$ be a chain complex of free, finitely generated 
$S$-modules.   Then, for all integers $q$,
\begin{equation}
\label{eq:vw}
\bigcup_{i\le q} \WW^i_1(E) = \bigcup_{i\le q} \VV^i_1(E).
\end{equation}
\end{theorem}

\begin{proof}
If $q\le -1$,  we have 
$\bigcup_{i\le q} \WW^i_1(E) = \bigcup_{i\le q} \VV^i_1(E)=\emptyset$. 
The  case $q\ge -1$ is done by induction on $q$, starting at $q=-1$.

Let $\m$ be a maximal ideal in $S$, and let 
$\rho\colon S \to S/\m$ be the corresponding $\k$-algebra morphism. 
Consider the K\"{u}nneth spectral sequence associated to the 
free chain complex $E\otimes_{S} S_{\m}$ and the change-of-rings 
map $\rho_{\m}\colon S_{\m}\to S_{\m}/\m S_{\m}$: 
\begin{equation}
\label{eq:ss local}
E^2_{s,t}=\Tor^{S_{\m}}_s (H_t(E)_{\m},S_{\m}/\m S_{\m})
\Rightarrow H_{s+t}(E\otimes_{S} S/\m).
\end{equation}

First suppose $\m \not\in\bigcup_{i=0}^{q} \WW^{i}_1(E)$.  Then, 
for each $t\le q$, we have $H_t(E)_{\m}=0$, and so 
$E^2_{s,t}=0$, and thus $E^{\infty}_{s,t}=0$. 
Consequently, $H_{i}(E\otimes_{S} S/{\m})=0$, for all $i\le q$, 
which means that 
$\m \not\in\bigcup_{i=0}^{q} \VV^i_1(E)$.

Now suppose $\m \in\bigcup_{i=0}^{q} \WW^{i}_1(E)$. 
We may assume $\m \not\in \supp H_{t}(E)$, 
for all $t<q$, and $\m \in \supp H_{q}(E)$, for 
otherwise we'd be done, by the induction hypothesis. 
These assumptions mean that $H_{t}(E)_{\m}=0$, for $t< q$, 
and $H_{q}(E)_{\m}\ne 0$. 
In particular, we have $E^2_{s,t}=0$, for $t< q$, and thus
$E^{2}_{0,q} = E^{\infty}_{0,q}$.  Hence, 
\begin{equation}
\label{eq:hkesm}
\begin{split}
H_{q}(E\otimes_{S} S/{\m})=E^{\infty}_{0,q}=E^{2}_{0,q}=
\hspace{1.2in}\\
H_{q}(E)_{\m}\otimes_{S_{\m}} S_{\m}/\m S_{\m}=
H_{q}(E)_{\m}/\m H_{q}(E)_{\m}.
\end{split}
\end{equation}

By assumption, $E_q$ is a finitely generated $S$-module;   
thus, $H_{q}(E)$ is also finitely generated.  Hence, 
$H_{q}(E)_{\m}$ is a non-zero, finitely generated $S_{\m}$-module. 
By Nakayama's Lemma, the module $H_{q}(E)_{\m}/\m H_{q}(E)_{\m}$ 
is also non-zero. Using \eqref{eq:hkesm}, we conclude that 
$\m \in\bigcup_{i=0}^{q} \VV^i_1(E)$, and this completes the proof. 
\end{proof}

Note that the freeness assumption is again crucial for this 
theorem to hold.  For instance, if $E$ is the chain complex from 
Example \ref{ex:aug}, then $\bigcup_{i\le 1} \VV^i_1(E)=\k\setminus \{0\}$, 
whereas $\bigcup_{i\le 1} \WW^i_1(E)=\k$.

\section{The resonance varieties of a graded algebra}
\label{sect:res cga}

We now turn to the resonance varieties associated to a 
commutative graded algebra, and set up a parameter 
space where vanishing resonance is a Zariski open condition. 

\subsection{Resonance varieties}
\label{subsec:resvar}

Let $A$ be a commutative graded algebra over a field $\k$, 
for short, a cga.  We will assume throughout that $A$ is 
connected, i.e., $A^0=\k$. 

Let $a\in A^1$, and assume $a^2=0$ (this condition 
is redundant if $\ch (\k)\ne 2$, by graded-commutativity 
of the multiplication in $A$).  
The {\em Aomoto complex}\/ of $A$ (with respect to $a\in A^1$) 
is the cochain complex of $\k$-vector spaces,
\begin{equation}
\label{eq:aomoto}
\xymatrixcolsep{30pt}
\xymatrix{(A , \delta(a))\colon  \ 
A^0\ar^(.66){\delta^0(a)}[r] & A^1\ar^(.45){\delta^1(a)}[r] 
& A^2  \ar^(.45){\delta^2(a)}[r]& \cdots },
\end{equation}
with differentials given by $\delta^i(a) (b)=ab$, for $b\in A^i$.
We define the {\em resonance varieties}\/ of $A$ as
\begin{equation}
\label{eq:rida}
\RR^i_d(A) = \{ a \in A^1 \mid a^2=0 \text{ and }
\dim _{\k} H^i(A,\delta(a))\ge d\}.
\end{equation}

It follows at once from the definition that
$\RR^0_1(A) = \{ 0\}$ and $\RR^0_d(A) = \emptyset$ for $d>1$,
since $A$ is connected.

\begin{lemma}
\label{lem:res closed}
Suppose $A$ is locally finite, i.e., $\dim_{\k} A^i<\infty$, for all $i\ge 1$. 
Then the sets $\RR^i_d(A)$ are Zariski closed cones inside the 
affine space $A^1$.
\end{lemma}

\begin{proof}
Let $Z=\set{a\in A^1 \mid a^2=0}$; clearly, $Z$ is a 
Zariski closed cone in $A^1$.  
By definition, an element $a\in  Z$ belongs to 
the set $\RR=\RR^i_d(A)$ if and only if 
$\rank \delta^{i-1}(a) + \rank \delta^{i}(a) \le b_i -d$, 
where $b_i=\dim_{\k} A^i$.  
An argument as in Lemma \ref{lem:zcv} now shows that  
$\RR$ is a Zariski closed subset of $A^1$. 
Clearly, $a\in \RR$ if and only if $\lambda a \in \RR$, 
for all $\lambda\in \k^{\times}$; thus, $\RR$ is a 
cone in $A^1$.
\end{proof}

\subsection{A parameter space for graded algebras}
\label{subsec:parameter}

Let $B=\bigoplus_{i\ge 0} B^i$ be a graded $\k$-vector 
space, with $B^0=\k$ and $\dim_{\k} B<\infty$, encoded by the sequence
$\{ b_i= \dim_{\k}  B^i \}_{i\ge 1}$. 
Given these data, we define a parameter space 
for all commutative graded algebras whose 
underlying graded vector space is $B$, as follows:
\begin{equation}
\label{eq:param}
P_B=\{  \text{$A$ a cga} \mid \dim_{\k} A^i =\dim_{\k}  B^i 
\text{ for all $i\ge 0$} \}.
\end{equation}

It is readily seen that $P_B$ is an affine cone in the vector space 
\begin{equation}
\label{eq:hombi}
V_{B}=\bigoplus_{i,j\ge 1} \Hom_{\k}(B^i\otimes B^j, B^{i+j}),
\end{equation}
cut out by homogeneous  quadrics (corresponding to the associativity 
conditions for a cga), and linear equations (corresponding to 
graded-commutativity conditions for a cga). 

\begin{example}
\label{ex:param2}
Suppose $B=\k\oplus B^1\oplus B^2$.  In that case, a 
commutative graded algebra  
$A\in P_B$ corresponds to an anti-symmetric $\k$-linear map 
$B^1\otimes B^1\to B^2$.   Thus, $P_B$ 
is an affine space; in fact, 
$P_B=\Hom_{\k}(B^1\wedge B^1,B^2)$ if $\ch(\k)\ne 2$. 
\end{example} 

\subsection{Vanishing resonance}
\label{subsec:vanish}

We are now ready to state and prove the main result of 
this section.

\begin{theorem}
\label{thm:gen res}
Let $B$ be a connected, finite-dimensional, graded vector space over 
an algebraically closed field $\k$.  For each $i\ge 1$, let 
\begin{equation}
\label{eq:zerores}
U^i_B=\{ A \in P_B \mid \RR^i_1(A)\subseteq \{0\} \}
\end{equation}
be the set of commutative graded algebras whose underlying 
graded vector space is $B$ and whose degree $i$ 
resonance varieties are trivial.  Then $U^i_B$ is 
a Zariski open subset of $P_B$.
\end{theorem}

\begin{proof}
We need to show that, for each $i\ge 1$, the set 
\begin{equation}
\label{eq:nonzerores}
C^i_B=\{ A \in P_B \mid \text{there is $a\in B^1\setminus \{0\}$ 
such that $H^i(A,\delta(a)) \ne 0$}\}
\end{equation}
is Zariski closed.
To that end, consider the set 
\begin{equation}
\label{eq:tib}
\{ (A,a) \in P_B \times B^1 \mid a^2=0 \text{ and } H^i(A,\delta(a)) \ne 0\}.
\end{equation}

A look at the proof of Lemma \ref{lem:res closed} 
shows that this set is defined by bi-homogeneous 
equations. In particular, the same equations define a 
Zariski closed subset $T^i_B$ inside $P_B\times \PP(B^1)$.

Now consider the first-coordinate projection map 
$p\colon P_B\times \PP(B^1) \to P_B$.  Since $\PP(B^1)$  
is a complete variety, $p$ is a closed 
map.  On the other hand, $p(T^i_B)=C^i_B$, 
and this completes the proof.
\end{proof}

\begin{remark}
\label{rem:grass}
The approach from \cite{PS-resrep}, where we consider 
connected, finite-dimensional cgas $A$ 
over $\C$ with $A^{>2}=0$ and we  fix the dimensions of $A^1$ and 
$\ker(A^1\wedge A^1\to A^2)$,  
shows exactly when the vanishing resonance condition in degree $1$ is generic. 
Theorem 1.1 from \cite{PS-resrep} implies that 
$U^1_B \ne \emptyset$ if and only if $b_2 \ge 2b_1 -3$, for an
arbitrary connected, finite-dimensional graded $\C$--vector space $B$. 
\end{remark}

\section{The characteristic and resonance varieties of a CW-complex}
\label{sect:jump cw}

In this section, we present a topological context in which the 
support varieties, homology jump loci, and resonance varieties 
arise, and reduce their computation to the case of finite CW-complexes.

\subsection{Homology with twisted coefficients}
\label{subsec:coeff}

Let $X$ be a connected CW-complex.  
Without loss of generality, we may assume $X$ has a single 
$0$-cell, call it $x_0$.  Let $\pi=\pi_1(X,x_0)$ be the fundamental 
group of $X$, based at $x_0$. 

Let $p\colon \wX \to X$ be the universal cover of our CW-complex.   
The cell structure on $X$ lifts to a cell structure on $\wX$.  
Fixing a lift $\tilde{x}_0\in p^{-1}(x_0)$ identifies 
the fundamental group of $X$ with the 
group of deck transformations of $\wX$, which 
permute the cells.  Therefore, we may view the 
cellular cell complex $C_{\bullet}(\wX, \Z)$, 
with differential $\tilde\partial$, as a chain complex 
of free left modules over the group ring $\Z{\pi}$.  

Given a right $\Z{\pi}$-module $M$, consider the chain complex 
$C_{\bullet}(X,M):=M\otimes_{\Z{\pi}} C_{ \bullet}(\wX, \Z)$, 
with differential $ \id_M\otimes \tilde\partial$.  
The homology groups of $X$ with coefficients 
in $M$ are then defined as $H_i(X,M):= H_i(C_{\bullet}(X,M))$. 

Noteworthy is the following situation. 
Let $X^{\ab}\to X$ be the universal abelian cover of $X$, 
with group of deck transformations $\pi_{\ab}=H_1(X,\Z)$. 
The cellular chain complex $(C_{\bullet}(X^{\ab},\Z),\partial)$ is 
a chain complex of free modules over the commutative 
ring $\Z[\pi_{\ab}]$.  The homology groups 
$H_i(X^{\ab},\Z)=H_i(X, \Z[\pi_{\ab}])$ 
are $\Z[\pi_{\ab}]$-modules, called the 
{\em Alexander invariants}\/ of $X$.

\subsection{Homology jump loci}
\label{subsec:cvs}

Fix a field $\k$, and let  $\Hom(\pi,\k^{\times})$ be the group 
of characters of $\pi$, with values in the multiplicative group of 
units of $\k$.  Since $\k^{\times}$ is an abelian group, every 
character factors through the abelianization $\pi_{\ab}$, and so 
we may identify $\Hom(\pi,\k^{\times})$ with $\Hom(\pi_{\ab},\k^{\times})$. 
Given a homomorphism 
$\rho\colon \pi\to \k^{\times}$, let $\k_{\rho}$ be the rank~$1$ 
local system on $X$ defined by $\rho$, and let $H_*(X, \k_{\rho})$ 
be the resulting twisted homology groups. 

There are three types of (co)homology jumping loci traditionally 
associated to a CW-complex $X$ as above. First, the 
{\em characteristic varieties}\/ of $X$ (with coefficients in $\k$) 
are the sets
\begin{equation}
\label{eq:char var}
\VV^i_d(X,\k) = \{\rho \in \Hom(\pi,\k^{\times}) \mid 
\dim_{\k} H_i(X, \k_{\rho}) \ge d\}.
\end{equation}

Second, the {\em Alexander varieties}\/ of $X$ (with coefficients 
in $\k$) are the supports of the exterior powers of the 
Alexander invariants of $X$,
\begin{equation}
\label{eq:alex var}
\WW^i_d(X,\k) = \supp \Big(\bigwedge^d H_i(X^{\ab},\k)\Big). 
\end{equation}

And finally, the {\em resonance varieties}\/ of $X$ (with coefficients 
in $\k$) are the jumping loci associated to the cohomology 
algebra $H^*(X,\k)$,
\begin{equation}
\label{eq:res vars}
\RR^i_d(X,\k) = \{ a \in H^1(X,\k) \mid a^2=0 \text{ and }
\dim _{\k} H^i(H^*(X,\k),\delta(a))\ge d\}.
\end{equation}

In absolute generality, there is not much structure on these sets.  
To remedy this situation, we need to impose some finiteness  
restrictions on $X$ in order to be able to say more about its jump loci. 
To start with, let us assume that $X$ has finite $1$-skeleton, 
and $\k$ is algebraically closed. Then the fundamental group 
$\pi=\pi_1(X,x_0)$ is finitely generated, and the character group 
$\Hom(\pi,\k^{\times})$ is an affine algebraic group, with coordinate 
ring the group algebra of the abelianization, $S=\k[\pi_{\ab}]$.  
 
Every character $\rho\colon \pi_{\ab} \to \k^{\times}$ extends 
$\k$-linearly to a ring morphism, $\bar\rho\colon \k[\pi_{\ab}] \to \k$, 
and thus gives rise to a maximal ideal $\m=\ker(\bar\rho)$ of $S$. 
Conversely, since $\k$ is algebraically closed, each maximal 
ideal $\m\in S$ determines a character 
$\rho\colon \pi_{\ab} \to \k^{\times}$.  Thus, 
\begin{equation}
\label{eq:specm}
\specm(\k[\pi_{\ab}]) =\Hom(\pi_{\ab},\k^{\times})=\Hom(\pi,\k^{\times}).
\end{equation}

Let $C_{\bullet} = (C_{\bullet}(X^{\ab},\k),\partial)$ be the equivariant 
chain complex of the universal abelian cover, with coefficients in $\k$. 
It is clear from the definitions that 
$\VV^i_d(X,\k)=\VV^i_d(C_{\bullet})$ and 
$\WW^i_d(X,\k)=\WW^i_d(C_{\bullet})$.  

Now suppose  that $X$ has finite $k$-skeleton, 
for some $k\ge 1$.  Then, for each $i<k$ and $d>0$, 
the sets $\WW^i_d(X,\k)$ and $\VV^i_d(X,\k)$ 
are Zariski closed subsets of the character group 
$\Hom(\pi,\k^{\times})$, either by definition, for the former, 
or by Lemma \ref{lem:zcv}, for the latter. Likewise, 
the sets $\RR^i_d(X,\k)$ are Zariski closed subsets 
of the affine space $H^1(X,\k)$, by Lemma \ref{lem:res closed}.  
In fact, as we shall see next, these statements also hold for $i=k$. 

\subsection{Reducing to the finite-dimensional case}
\label{subsec:finite dim}

For the purpose of computing resonance varieties or 
homology with certain twisted coefficients, in the 
finiteness range of $X$, the following proposition allows 
us to replace $X$ by a {\em finite}\/ CW-complex.

\begin{prop}
\label{prop:skeleton}
Let $X$ be a connected CW-complex with finite $k$-skeleton, 
for some $k\ge 1$.  There exists then a finite CW-complex $Y$ 
of dimension at most $k+1$, with $Y^{(k)}=X^{(k)}$, and a map 
$f\colon Y\to X$ such that the following hold.
\begin{enumerate}
\item \label{f1}  The induced homomorphism, 
$f_*\colon H_1(Y,\Z)\to H_1(X,\Z)$, is an isomorphism.

\item \label{f2}  For every commutative ring $R$, 
and for every homomorphism $\rho\colon H_1(X,\Z) \to R^{\times}$, 
the induced homomorphism, 
$f_*\colon H_i(Y, R_{\rho\circ f_*})\to H_i (X,R_{\rho})$, 
is an $R$-module isomorphism, for all $i\le k$.

\item \label{f3}  The isomorphism $f^*\colon \Hom(H_1(X,\Z),\k^{\times}) 
\to  \Hom(H_1(Y,\Z),\k^{\times})$ 
restricts to isomorphisms $\WW^i_d(X,\k) \to \WW^i_d(Y,\k)$,   
for all $i\le k$. 

\item \label{f4}  The isomorphism $f^*\colon \Hom(H_1(X,\Z),\k^{\times}) 
\to  \Hom(H_1(Y,\Z),\k^{\times})$ 
restricts to isomorphisms $\VV^i_d(X,\k) \to \VV^i_d(Y,\k)$,   
for all $i\le k$. 

\item \label{f5}  The isomorphism $f^*\colon H^1(X,\k)\to H^1(Y,\k)$ 
restricts to isomorphisms $\RR^i_d(X,\k) \to \RR^i_d(Y,\k)$,  
for all $i\le k$. 
\end{enumerate}
\end{prop}

\begin{proof}
The cellular chain complex $(C_i(X^{\ab}),\partial^{\ab}_i)_{i\ge 0}$ 
is a chain complex of free modules over the (commutative) 
Noetherian ring $S=\Z[{\pi}_{\ab}]$.  Since $C_k(X^{\ab})$ is 
finitely generated as an $S$-module, the $S$-submodule 
$B_k(X^{\ab})=\im (\partial^{\ab}_{k+1})$ is also finitely 
generated, let's say, by the images of the $(k+1)$-cells 
$e_1,\dots, e_r$. Set 
\[
Z:=X^{(k)} \cup e_1 \cup \dots \cup e_r.
\]

Clearly, $Z$ is a finite subcomplex of $X^{(k+1)}$; 
let $g\colon Z\to X$ be the inclusion map. 
Consider the cup-product map, 
$\cup^{k+1}_X\colon H^1(X,\k)\otimes H^k(X,\k)\to H^{k+1}(X,\k)$. 
Passing to $\k$-duals, we obtain the comultiplication map, 
$\nabla_{k+1}^X \colon H_{k+1}(X,\k) \to H_1(X,\k)\otimes H_k(X,\k)$. 

The dual of $\im (\cup_X)$ may be identified with $\im(\nabla_X)$. 
Proceeding in the same fashion with the space $Z$, and 
comparing the resulting maps via the induced homomorphism 
$g_*\colon H_*(Z,\k)\to H_*(X,\k)$, we obtain the following 
commuting diagram:
\begin{equation}
\label{eq:nabla diag}
\xymatrix{
H_{k+1}(Z,\k) \ar@{->>}[r] \ar^{g_*}[d] \ar@/^1.5pc/^{\nabla_{k+1}^Z} [rr]
& \im(\nabla_{k+1}^Z) \ar@{^{(}->}[r] \ar^{g_*}[d]  
& H_1(Z,\k)\otimes H_k(Z,\k) \ar^{g_*\otimes g_*}[d] \\
H_{k+1}(X,\k) \ar@{->>}[r] \ar@/_2pc/^{\nabla_{k+1}^X} [rr] 
& \im(\nabla_{k+1}^X) 
\ar@{^{(}->}[r] & H_1(X,\k)\otimes H_k(X,\k).
}
\end{equation}
\smallskip

In general, the vertical arrow in the center of the above 
diagram is not surjective. So pick a (finite) $\k$-basis, 
$x_1, \dots , x_s$ for $ \im(\nabla_{k+1}^X)$, 
lift those homology classes back to 
$y_1=[z_1],\dots ,y_s=[z_s] \in H_{k+1}(X,\k)$, 
and represent the cycles $z_1,\dots ,z_s$ as $\k$-linear 
combinations of finitely many $(k+1)$-cells of $X$, say, 
$e_{r+1},\dots, e_{t}$. Then, the CW-complex 
\[
Y:=Z \cup e_{r+1}\cup \dots \cup e_{t}
\]
is again a finite subcomplex of $X^{(k+1)}$; let $f\colon Y\to X$ 
be the inclusion map.  Redrawing diagram \eqref{eq:nabla diag} 
with $Z$ replaced by $Y$, $g_*$ by $f_*$, etc, we see 
that the middle arrow is now surjective.  Hence, the 
dual map, $f^*\colon \im (\cup^{k+1}_X) \to \im(\cup^{k+1}_Y)$, 
is injective.

Clearly, $f_*\colon H_1(Y,\Z)\to H_1(X,\Z)$ is an isomorphism, 
thus proving part \eqref{f1}, and showing that the maps $f^*$ 
in parts \eqref{f3}--\eqref{f5} are also isomorphisms. 

By construction, the map $f_*\colon H_i(Y,R)\to H_i(X,R)$
from part \eqref{f2}  
is an $R$-isomorphism, for  all $i\le k$, since 
$\rho$ corresponds to a change-of-rings map, $S\to R$. 
Furthermore, parts \eqref{f3} and \eqref{f4} are 
now a direct consequence of part \eqref{f2}.

Also by construction, the map $f^*\colon H^1(X,\k)\to H^1(Y,\k)$
restricts to isomorphisms $\RR^i_d(X,\k) \isom \RR^i_d(Y,\k)$,  
for all $i\le k$ and all $d>0$.  This finishes the proof of part \eqref{f5}.
\end{proof}

The following corollary is now immediate. 

\begin{corollary}
\label{cor:vwr}
Let $X$ be a connected CW-complex with finite $k$-skeleton, 
for some $k\ge 1$.  Then, for each $i\le k$ and $d>0$, 
the sets $\WW^i_d(X,\k)$ and $\VV^i_d(X,\k)$ 
are subvarieties of the character group 
$\Hom(\pi,\k^{\times})$, while 
the sets $\RR^i_d(X,\k)$ are subvarieties 
of the affine space $H^1(X,\k)$.
\end{corollary}

Using Theorem \ref{thm:vv ww} and Proposition \ref{prop:skeleton}, 
we obtain the following result, which recovers Corollary 3.7 
from \cite{PS-plms}. 

\begin{corollary}
\label{cor:vv ww}
Let $X$ be a connected CW-complex with finite $k$-skeleton, 
for some $k\ge 1$.  Then, for all integers $q\le k$,
\begin{equation}
\label{eq:vwx}
\bigcup_{i\le q} \VV^i_1(X,\k) = \bigcup_{i\le q} \WW^i_1(X,\k).
\end{equation}
\end{corollary}

\section{The equivariant spectral sequence}
\label{sec:equiv ss}

We now relate the homology jump loci associated 
to the first page of the equivariant spectral sequence of  
a space to the resonance varieties of its cohomology 
algebra.

\subsection{The spectrum of the associated graded ring of a group ring}
\label{subsec:jumps ss}

Let $G$ be a finitely generated abelian group, and let 
$\k{G}$ be the group-ring over an algebraically closed field $\k$.  
The powers of the augmentation ideal, 
$J=\ker(\k{G}\xrightarrow{\epsilon}\k)$, define 
a descending filtration on $\k{G}$. The associated 
graded ring, $S^{\hdot}=\gr_{J}^{\hdot}(\k{G})$, is an affine 
$\k$-algebra, whose maximal spectrum (endowed with the 
Zariski topology) we denote by $\specm(S)$.

Let $\bar{G}=G/\text{tors}(G)$ be the maximal torsion-free 
quotient of $G$, and let  $\bar{J}$ be the augmentation ideal 
of $\k{\bar{G}}$. The associated graded ring, 
$\bar{S}=\gr_{\bar{J}}(\k{\bar{G}})$, may be identified 
with the polynomial ring $\k[x_1,\dots, x_r]$, where $r=\rank(\bar{G})$. 
Consequently, $\specm(\bar{S})$ may be identified with 
the affine space $H^1(\bar{G},\k)=\k^r$.

The natural projection $\phi\colon G\surj \bar{G}$ extends to a 
ring epimorphism $\phi\colon S\surj \bar{S}$. In general, this  
morphism is not injective.  For instance, if $\ch \k=p$ and 
$G=\Z/p^s\Z$, then $\k{G}= \k[t]/(t^{p^s}-1)= \k[t]/(t-1)^{p^s}$, 
and thus $S=\k[x]/(x^{p^s})$, whereas $\bar{S}=\k$.  
Passing to maximal spectra, though, fixes this fat point issue.

Let $\nil (S)$ be the nilradical of the ring $S$, i.e., the intersection 
of all its prime ideals. 

\begin{lemma}
\label{lem:sbars}
The kernel of the epimorphism $\phi\colon S\surj \bar{S}$ is 
equal to $\nil (S)$.   Consequently, $\phi$ induces a homeomorphism 
$\specm(\bar{S}) \to \specm(S)$. 
\end{lemma}

\begin{proof}
First suppose that $G=G_1\times G_2$.  
Recall that the group-ring of a direct product of groups 
is canonically isomorphic to the tensor product of the group-rings 
of the factors.  Thus, we have an isomorphism of $\k$-algebras, 
$\alpha\colon \k{G_1} \otimes_{\k} \k{G_2} \isom \k{G}$.  
Now let $J\subset \k{G}$ and $J_i\subset \k{G_i}$ be the respective 
augmentation ideals.   A standard inductive argument shows that 
\begin{equation}
\label{eq:alpha}
\alpha\Big( \sum_{s+t = n}  J_1^s \otimes J_2^t\Big)=J^n,
\end{equation}
for all $n\ge 0$. Hence, the map $\alpha$ 
induces an isomorphism of (connected) graded $\k$-algebras, 
$\gr(\alpha)\colon S_1^{\hdot} \otimes_{\k} S_2^{\hdot} \isom S^{\hdot}$.
Applying the same argument to the decomposition 
$\bar{G}=\bar{G}_1\times \bar{G}_2$, we obtain  
an isomorphism 
$\gr(\bar\alpha)\colon \bar{S}_1^{\hdot}  \otimes_{\k} \bar{S}_2^{\hdot}
\isom \bar{S}^{\hdot}$ which fits into the commuting diagram 
\begin{equation}
\label{eq:specs}
\xymatrixcolsep{30pt}
\xymatrixrowsep{22pt}
\xymatrix{
S_1 \otimes_{\k} S_2 \ar^(.6){\gr(\alpha)}[r] \ar[d]& S\ar[d]\\
\bar{S}_1 \otimes_{\k} \bar{S}_2 
\ar^(.6){\gr(\bar\alpha)}[r] & \,\bar{S} \, .
}
\end{equation}

Taking $G_1=\bar{G}$ and $G_2= \Tors (G)$, we infer from \eqref{eq:specs}
that $\phi$ may be identified with 
$\id \otimes \epsilon \colon \bar{S} \otimes S_2 \to \bar{S} \otimes \k$.
In particular, $\ker (\phi)$ is identified with $\bar{S} \otimes S_2^+$. 

Plainly, $\nil (\bar{S})=0$, and so $\nil (\bar{S} \otimes S_2) \subseteq \ker (\phi)$.  
We have to prove the reverse inclusion. Note that $S_2$ is a tensor product
of graded algebras of the form $R^{\hdot}= \gr_{J}^{\hdot} \k (\Z/p^s\Z)$,
for some prime $p$ and integer $s\ge 1$. Clearly, it is enough to show that 
$R^+\subseteq \nil (R)$. 

If $p\ne \ch(\k)$, then $R^1=(\Z/p^s\Z) \otimes \k$ 
vanishes, and thus $R^+=0$.  If $p= \ch(\k)$, then 
$R=\k[x]/(x^{p^s})$, and thus $R^+\subseteq \nil (R)$. 
In either case, the desired conclusion holds, and we are done.
\end{proof}

\subsection{The first page of the equivariant spectral sequence}
\label{subsec:equiv ss}
Let $X$ be a connected, finite-type CW-complex, 
and fix a coefficient field $\k$.  As before, let $\cup_X$ be 
the cup product map in $H^*(X,\k)$, and let $\nabla^X$ be the 
co-multiplication map in $H_*(X,\k)$.   

Next, let $\nu\colon \pi\surj G$ be an epimorphism from the 
fundamental group $\pi=\pi_1(X,x_0)$ to an abelian group $G$ 
(necessarily, $G$ must be finitely generated).  Let $\k{G}$ be 
the group-ring of $G$, and let $\widehat{\k{G}}$ be the completion 
of this ring with respect to the filtration by powers of the augmentation 
ideal. Composing the completion map $\k{G}\to \widehat{\k{G}}$ 
with the extension of $\nu$ to group rings yields a ring morphism, 
$\hat{\nu}\colon \k{\pi}\to \widehat{\k{G}}$.  Clearly, this morphism 
makes the completion $\widehat{\k{G}}$ into a module 
over $\k{\pi}$. 

Following the setup from \cite{PS-tams}, let
\begin{equation}
\label{eq:e1 page}
E=(E^1(X,\widehat{\k{G}}),d^1)
\end{equation}
be the first page of the {\em equivariant spectral sequence} of $X$ 
with coefficients in $\widehat{\k{G}}$.   This is a chain complex 
of  free, finitely-generated modules over the
affine algebra $S=\gr_J(\k{G})$.  The $i$-th term of this chain complex 
is 
\begin{equation}
\label{eq:e1i}
E_i=S\otimes_{\k} H_i(X,\k),
\end{equation}
while the $i$-th differential, 
$d^1_i\colon E_i\to E_{i-1}$, when restricted to the generating set 
$\{1\}\otimes H_i(X, \k)$, is the composite
\begin{equation}
\label{eq:d1}
\xymatrixcolsep{28pt}
d^1_i\colon\! \xymatrix{H_i(X, \k) \ar^(.36){\nabla^X_i}[r] 
& H_1(X, \k)\otimes_{\k} H_{i-1}(X, \k)
\ar^{\nu_{*} \otimes \id}[r] & H_1(G, \k)\otimes_{\k} H_{i-1}(X, \k)}.
\end{equation}

By definition, the transpose of $\nabla^X_i$ is the cup-product map 
$\cup_X^i$.   Thus, the transpose of $d^1_i$ is the composite
\begin{equation}
\label{eq:d1 tr}
\xymatrixcolsep{28pt}
(d^1_i)^{\top} \colon\! \xymatrix{
H^1(G, \k)\otimes_{\k} H^{i-1}(X, \k)
\ar^(.48){\nu^{*} \otimes \id}[r] 
& H^1(X, \k)\otimes_{\k} H^{i-1}(X, \k)
\ar^(.65){\cup_X^i}[r] & H^i(X, \k) 
}.
\end{equation}

\subsection{Homology jump loci and resonance varieties}
\label{subsec:cv rv}

Now assume $\k$ is algebraically closed.
By Lemma \ref{lem:sbars}, we may identify the maximal 
spectrum $\specm(S)$ with $H^1(\bar{G},\k)$, where 
recall $\bar{G}$ is the maximal torsion-free quotient of $G$.

Using this identification, we may view the homology jump loci 
$\VV^i_d(E)$ of the chain complex $E$ from \eqref{eq:e1 page}
as subsets of the $\k$-vector space $H^1(\bar{G},\k)$. 
The next result compares these loci with the 
resonance varieties of $X$, viewed as subsets 
of $H^1(X,\k)=H^1(\pi,\k)$. 

\begin{theorem}
\label{thm:cv res}
Let $\nu\colon \pi\surj G$ be an epimorphism
onto an abelian group, 
and set $\bar{\nu}=\phi\circ \nu \colon \pi\surj \bar{G}$.  
Let $w\in H^1(\bar{G},\k)$.  Then 
\begin{equation}
\label{eq:vr}
w \in \VV^i_d(E) \same \bar{\nu}^* (w) \in \RR^i_d(X,\k).
\end{equation}
\end{theorem}

\begin{proof}
By definition, an element $a\in H^1(X,\k)$ 
belongs to  $\RR^i_d(X,\k)$ if and only if 
$a^2=0$ and $\dim_{\k} H^i(H^*(X,\k), \delta(a) )\ge d$, 
where recall $\delta^i(a)\colon H^i(X,\k)\to H^{i+1}(X,\k)$ 
is left-multiplication by $a$.

Now let $w\in H^1(\bar{G},\k)$. 
Clearly, $w^2=0$, and thus $\bar{\nu}^* (w)^2=0$. 
Let $\rho\colon S \to \k$ be the 
$\k$-algebra morphism defined by $\phi^*(w)\in \specm(S)$. 
Denote by $d^1(w)$ the specialization of $d^1$ at $\rho$. 
Using sequence \eqref{eq:d1 tr}, we find that 
\begin{equation}
(d^1_i(w))^{\top}=\delta^{i-1}(\bar{\nu}^* (w)).
\end{equation}

Therefore,
\begin{align*}
w \in \VV^i_d(E) &\same \dim_\k H_i(H_*(X,\k), d^1(w))\ge d
\\
& \same \dim_\k H^i(H^*(X,\k), (d^1(w))^{\top})\ge d 
\\
 & \same \bar{\nu}^* (w) \in \RR^i_d(X,\k),
\end{align*}
and this completes the proof.
\end{proof}

\begin{remark}
\label{rem:vr}
Noteworthy is the situation when $\pi_{\ab}$ is torsion-free 
and $\nu$ is the abelianization map, $\ab\colon \pi\surj \pi_{\ab}$, 
in which case  $H^1(\bar{\pi}_{\ab},\k)= H^1(X,\k)$. 
Let $E$ be the first page of the corresponding spectral 
sequence.  Applying Theorem \ref{thm:cv res}, we may 
then identify  $\VV^i_d(E)$ with $\RR^i_d(X,\k)$. 
\end{remark}

\begin{remark}
\label{rem:next pages}
It would be interesting to find a similar interpretation for the 
homology jumping loci of the other pages in the equivariant 
spectral sequence.  Such an interpretation would likely 
involve the higher-order Massey products in $H^*(X,\k)$.
This is suggested by the case $G=\Z$, when, as explained 
in  \cite[\S1.4]{PS-tams}, our equivariant spectral sequence 
is related to the Farber--Novikov spectral sequence, whose 
differentials are given by certain Massey products, see 
for instance \cite{F}. 
\end{remark}

\section{Jump loci and finiteness properties}
\label{sect:jumps}

In this section, we relate the vanishing of the resonance 
varieties of a space to the finiteness properties of its  
completed Alexander-type invariants.  First, we need 
to recall a well-known fact from commutative algebra.  
For the reader's convenience, we give a sketch of a proof 
(see also \cite[Proposition 9.3]{SYZ}).

\begin{lemma}
\label{lem:findim}
Let $\k$ be an algebraically closed field, and 
let $M$ be a finitely generated module over an affine 
$\k$-algebra $S$.  Then:
\[
 \dim_{\k} M < \infty \same \text{$\supp M$ is finite}.
\]
\end{lemma}

\begin{proof}
The module $M$ is finite-dimensional over $\k$ 
if and only if $M$ has finite length.  This condition 
is equivalent to $M$ having $0$ Krull 
dimension, i.e., $\dim (S/\ann(M))=0$.  In turn, this 
means that the support of $M$ is $0$-dimensional, 
i.e., $\supp M$ is finite.
\end{proof}

Now let $X$ be a connected CW-complex 
with finite $k$-skeleton, for some $k\ge 1$.  
Given an epimorphism $\nu\colon \pi\surj G$ from 
$\pi=\pi_1(X,x_0)$ to a (finitely generated) abelian 
group $G$, let $\phi\colon G\surj \bar{G}$ be the 
projection onto the maximal free-abelian quotient of $G$, 
and set $\bar{\nu}=\phi \circ \nu$.

Let $\k{G}$ be the group-ring of $G$, with coefficients in an 
algebraically closed field  $\k$, let $J$ be its augmentation 
ideal, and let $\widehat{\k{G}}$ be the $J$-adic completion 
of $\k{G}$. As shown in Lemma \ref{lem:sbars}, the associated 
graded ring $S=\gr_J(\k{G})$ is an affine $\k$-algebra, 
whose  maximal spectrum $\specm( S)$ may be identified 
with $H^1(\bar{G},\k)$.

\begin{theorem}
\label{thm:nustar res}
Let  $X^{\nu}\to X$ be the Galois cover associated to an 
epimorphism $\nu\colon \pi\surj G$, and let $\widehat{H_*(X^{\nu},\k)}$ 
be the $J$-adic completion of the $\k{G}$-module 
$H_*(X^{\nu},\k)= H_*(X, (\k{G})_{\nu})$. Then 
\[
\bar\nu^*(H^1(\bar{G},\k))\cap 
\Big( \bigcup_{0\le i\le k} \RR^i_1(X,\k)\Big) =\{0\} 
\implies 
\dim_{\k} \bigoplus_{0\le i\le k} \widehat{H_i(X^{\nu},\k)} < \infty .
\]
\end{theorem}

\begin{proof}
By Proposition \ref{prop:skeleton}, we may assume $X$ is 
a finite complex.
We use the equivariant spectral sequence $E^{\bullet}(X,\widehat{\k{G}})$  
associated to the Galois cover  $X^{\nu}\to X$ corresponding to 
the epimorphism $\nu\colon \pi_1(X)\surj G$,  with coefficients given by 
the ring morphism $\hat\nu\colon \k{\pi_1(X)}\to \widehat{\k{G}}$. 

The first page of the spectral sequence is $(E^1,d^1)$, 
a chain complex of free, finitely generated $S$-modules, 
with $E^1=S\otimes_{\k} H_*(X,\k)$.  We then have:
\[
\begin{aligned}
  \bar\nu^*(H^1(\bar{G},\k))\cap \Big( \bigcup_{i\le k} \RR^i_1(X,\k)\Big) 
 & = \bar\nu^* \Big( \bigcup_{i\le k} \VV^i_1(E^1) \Big)
    & \text{by Theorem \ref{thm:cv res}} \\
     & = \bar\nu^* \Big(  \bigcup_{i\le k} \WW^i_1(E^1) \Big)
    & \text{by Theorem \ref{thm:vv ww}} \\
     & = \bar\nu^* \Big(  \supp\Big( \bigoplus_{i\le k} H_i(E^1) \Big) \Big)
    &
     \\
        & =  \bar\nu^* \Big( \supp\Big( \bigoplus_{i\le k} E^2_i\Big) \Big).
    &
\end{aligned}
\]

Using our assumption and the injectivity of $\bar\nu^*$, 
we conclude that the support of $\bigoplus_{i\le k} E^2_i$ 
is the set $\{0\}$. Hence, by Lemma \ref{lem:findim}, 
the $\k$-vector space $\bigoplus_{i\le k} E^{2}_i$ is 
finite-dimensional.  It  follows that $\bigoplus_{i\le k} E^{\infty}_i$ 
is also finite-dimensional. 

On the other hand, as shown in \cite[\S{8}]{PS-tams}, the equivariant 
spectral sequence converges to $\widehat{H_*(X^{\nu},\k)}$, 
and the spectral sequence filtration on the limit is separated.  
The desired conclusion readily follows. 
\end{proof}

\begin{corollary}
\label{cor:univ abel}
If $\RR^i_1(X,\k)\subseteq \{0\}$ for $1\le i\le k$, then 
$\dim_{\k} \widehat{H_i(X^{\ab},\k)} < \infty$ for $i\le k$.
\end{corollary}

In the particular case when $k=1$ and $\k =\C$, this corollary 
was proved in \cite[Theorem C]{DP-ann} by different methods, 
specific to homological degree $1$. It is also shown in \cite{DP-ann} 
that the converse holds in this case, under an additional 
formality assumption.  More precisely, if $X$ is $1$-formal 
in the sense of Sullivan \cite{S}, and the completion of 
$H_1(X^{\ab},\C)$ is finite-dimensional, then 
$\RR^1_1(X,\C)\subseteq \{0\}$. 

In general, though, the converse to  Corollary \ref{cor:univ abel} 
does not hold.   For instance, if $X$ is the Heisenberg 
$3$-dimensional nilmanifold, then the cup-product vanishes 
on $H^1(X,\C)=\C^2$, and so $\RR^1_1(X,\C)= \C^2$. 
On the other hand, $X^{\ab}\simeq S^1$, and thus  
$H_1(X^{\ab},\C)=\C$.

It should also be pointed out that it is really necessary 
to take the completion of the Alexander invariant in 
Corollary \ref{cor:univ abel}, even when $X$ is formal.
For instance, if $X$ is the presentation $2$-complex for 
the group $\pi=\langle a, b \mid a^2 b =ba^2\rangle$ 
from \cite[Example 6.4]{PS-plms}, 
we have that $\RR^1_1(X,\C)= \{0\}$, yet
$\dim_{\C} H_1(X^{\ab},\C) =\infty$.

For more details on the relationship between the 
first resonance variety and the first Alexander 
invariant of a space, we refer to \cite{PS-jtop}.  
Further generalizations and applications will 
be pursued elsewhere.

\begin{ack}
This work was started in May-June 2010, when both authors 
visited the Centro di Ricerca Matematica Ennio De Giorgi in Pisa, 
and was continued in Spring 2011, when the first author visited 
Northeastern University.  The work was completed while the 
authors visited the University of Sydney, in December 2012. 
We are grateful for the support and hospitality provided, 
and we also thank Krzysztof Kurdyka and 
Lauren\c{t}iu P\u{a}unescu for helpful discussions. 
We are also grateful to the referees, for their pertinent 
remarks and suggestions.
\end{ack}

\newcommand{\arxiv}[1]
{\texttt{\href{http://arxiv.org/abs/#1}{arxiv:#1}}}
\newcommand{\arx}[1]
{\texttt{\href{http://arxiv.org/abs/#1}{arXiv:}}
\texttt{\href{http://arxiv.org/abs/#1}{#1}}}
\newcommand{\doi}[1]
{\texttt{\href{http://dx.doi.org/#1}{doi:#1}}}
\renewcommand{\MR}[1]
{\href{http://www.ams.org/mathscinet-getitem?mr=#1}{MR#1}}

\end{document}